\documentclass[a4paper]{amsart}

\usepackage[OT2,T1]{fontenc}
\DeclareSymbolFont{cyrletters}{OT2}{wncyr}{m}{n}
\DeclareMathSymbol{\Sha}{\mathalpha}{cyrletters}{"58}

\usepackage{amssymb,amscd}
\usepackage{mathrsfs}

\usepackage[
backref,
pdfauthor={Han Wu},  % add other authors
pdftitle={Arithmetic of Ch\^atelet surfaces under extensions of base fields},
]{hyperref}
\usepackage[alphabetic,backrefs,lite]{amsrefs}  % for bibliography, can choose nobysame

\theoremstyle{plain}
\theoremstyle{definition}

\newtheorem{theorem}{Theorem}[section]
\newtheorem{thm}{Theorem}[section]
\newtheorem{lemma}[theorem]{Lemma}

\newtheorem{proposition}[theorem]{Proposition}

\theoremstyle{definition}

\newtheorem{eg}[theorem]{Example}

\theoremstyle{remark}
\newtheorem{remark}[theorem]{Remark}

\renewcommand{\AA}{\mathbb{A}}

\newcommand{\QQ}{\mathbb{Q}}

\newcommand{\ZZ}{\mathbb{Z}}

\newcommand{\FF}{\mathbb{F}}
\newcommand{\GG}{\mathbb{G}}
\newcommand{\PP}{\mathbb{P}}

\newcommand{\Ocal}{{\mathcal O}}

\newcommand{\half}{1/2}

\newcommand{\inuap}{\inv_{v'}(A(P_{v'}))}
\newcommand{\invap}{\inv_v(A(P_v))}
\newcommand{\invaq}{\inv_v(A(Q_v))}
\newcommand{\invaqo}{\inv_v(A(Q_0))}

\DeclareMathOperator{\inv}{inv}

\newcommand{\defi}[1]{\textsf{#1}} % for defined terms

\newcommand{\Br}{\textup{Br}}
\newcommand{\et}{\textup{\'et}}

\makeatletter
\g@addto@macro\bfseries{\boldmath}  % This makes math in section titles bold.
\makeatother

\begin{document}
	
	\begin{title}
		{Arithmetic of Ch\^atelet surfaces under extensions of base fields}  %\'etale
	\end{title}
	\author{Han Wu}
	\address{University of Science and Technology of China,
		School of Mathematical Sciences,
		No.96, JinZhai Road, Baohe District, Hefei,
		Anhui, 230026. P.R.China.}
	\email{wuhan90@mail.ustc.edu.cn}
	\date{}
	%\thanks{The author was partially supported by USTC}
	% 11G05, , 14H25, 14H52, 14K15, 14J30

	%\thanks{The authors were partially supported by University of Science and Technology of China}
	%\thanks{\textit{MSC 2010} : 11G35 14G05  14G25 14J20}

	% % % ----------------------------------------------------------------------

	% % % ----------------------------------------------------------------------

	\begin{abstract} 
		For Ch\^atelet surfaces defined over number fields, we study two
		arithmetic properties, the Hasse principle and weak approximation, when
		passing to an extension of the base field. Generalizing a construction of Y.
		Liang, we show that for an arbitrary extension of number fields $L/K,$ there
		is a Ch\^atelet surface over $K$ which does not satisfy weak approximation over
		any intermediate field of $L/K,$ and a Ch\^atelet surface over $K$ which satisfies
		the Hasse principle over an intermediate field $L'$ if and only if $[L' : K]$ is
		even.
	\end{abstract} 
	
	\maketitle

	\section{Introduction}

	Throughout this paper, let $K$ be a number field, and let 
	$\Omega_K$ be the set of all nontrivial places of $K.$ For each $v\in \Omega_K,$ let $K_v$ denote the completion of $K$ at $v.$
	Let $S\subset \Omega_K$ be a finite subset. Let $\AA_K$ (respectively $\AA_K^S$) be the ring of ad\`eles (ad\`eles without $S$ components) of $K.$ We always assume that a field $L$ is a finite extension of $K.$ Let $S_L\subset \Omega_L$ denote the set of places of $L$ lying over places in $S.$

	Let $X$ be a proper algebraic variety over $K.$ The set $X(K)$ of $K$-rational
	points of $X$ can be viewed as a subset of the set $X(\AA_K)$ of adelic points via
	the diagonal embedding. We say that $X$ is a \defi{counterexample to the Hasse principle} if $X(\AA_K)\neq\emptyset$ whereas $X(K)=\emptyset.$ By the properness of $X,$ the set of adelic points $X(\AA_K^S)$ can be identified with the product $\prod_{v\in \Omega_K\backslash S}X(K_v),$ and hence equipped with the product topology of $v$-adic topologies. We say that $X$
	satisfies \defi{weak approximation} (respectively \defi{weak approximation off $S$}) if $X(K)$ is dense in $X(\AA_K)$ (respectively in $X(\AA_K^S)$), cf. \cite[Chapter 5.1]{Sk01}.

		In this paper, we focus on the case where $X$ is a Ch\^atelet surface over $K,$ i.e., a smooth projective models of affine surface in $\AA_K^3$ defined by the equation
		\begin{equation}\label{equation}
			y^2-az^2=P(x),
		\end{equation} where $a\in K^\times,$ and $P(x)$ is a separable degree-$4$ polynomial in $K[x].$
 Over a special number field $K,$ the Hasse principle and weak approximation for
	Ch\^atelet surfaces (and many other varieties) have been studied in a lot of
	earlier papers (e.g. \cite{Is71}, \cite{CTSSD87a}, \cite{CTSSD87b}, \cite{Po09}, etc.). In \cite{Li18}, Liang
	pioneered the study of non-invariance of arithmetic
	properties under an extension of base fields and proved, among others, that
	for any number field $K,$ there is a Ch\^atelet surface $V$ over $K$ and a quadratic
	extension $L/K$	 such that $V(K)\neq \emptyset,$ $V$ satisfies weak approximation, but the
	base extension $V_L$ does not satisfies weak approximation off the archimedean places of $L.$ In this paper, we generalize Liang's construction and obtain
	further results that apply to an arbitrary extension $L/K.$
	
	Our main results are the following:
		\begin{thm}[Theorem \ref{thm: interesting result for Chatelet surface1}]
		Let  $L/K$ be any extension of number fields, and let $S \subset \Omega_K\backslash \{$all complex and $2$-adic places$\}$ be a finite nonempty subset such that every place in $S$ splits completely in $L.$
		
		Then, there exists a Ch\^atelet surface $V$ defined over $K$ with $V(K)\neq \emptyset,$ such that
		for every intermediate field $K\subset L'\subset L$ and every finite subset $T'\subset \Omega_{L'},$ the base extension $V_{L'}$ satisfies weak approximation off $T'$ if and only if  $T'\cap S_{L'}\neq \emptyset.$ In particular,  the surface $V_{L'}$ does not satisfy weak approximation for every $K\subset L'\subset L.$
	\end{thm}
	
	\begin{thm}[Theorem \ref{thm: interesting result for Chatelet surface2}]
		For any extension of number fields $L/K,$ there exists a Ch\^atelet surface $V$ over $K$ with $V(\AA_K)\neq \emptyset,$ such that
		for every intermediate field $K\subset L'\subset L,$
		\begin{itemize}
			\item If the degree $[L':K]$ is odd, then the surface $V_{L'}$ is a counterexample to the Hasse principle, i.e. $V(L')=\emptyset.$ In particular, the surface $V$ is a counterexample to the Hasse principle.
			\item If the degree $[L':K]$ is even, then the surface $V_{L'}$ satisfies weak approximation. In particular, in this case, the set $V(L')\neq \emptyset.$
		\end{itemize}
	\end{thm}

	To construct the Ch\^atelet surfaces in these theorems, the parameters in
	the equation (\ref{equation}), i.e. the element $a\in K$ and the coefficients of the polynomial
	$P(x),$ need to be chosen carefully using approximation theorems for the affine
	line and \v{C}ebotarev's density theorem. To verify the statements about weak
	approximation and the Hasse principle, we shall analyze the Brauer-Manin
		obstruction and use the well known theorem that for Ch\^atelet surfaces this
		obstruction (to weak approximation or the Hasse principle) is the only one
		(\cite{CTSSD87a}, \cite{CTSSD87b}).

	\section{Notation and preliminaries}
	\subsection{Notation} Given a number field $K,$ let $\Ocal_K$ be the ring of its integers. Let $\infty_K\subset \Omega_K$ be the subset of all archimedean places, and let $2_K\subset \Omega_K$ be the subset of all $2$-adic places. Let $\infty_K^r\subset \infty_K$ be the subset of all real places, and let $\infty_K^c\subset \infty_K$ be the subset of all complex places. Let $\Omega_K^f=\Omega_K\backslash \infty_K$ be the set of all finite places of $K.$ Let $K_v$ be the completion of $K$ at  $v\in \Omega_K.$ %and let $\AA_K$ be the ring of ad\`eles of $K.$ 	
	%If an element $a\in \Ocal_K$ is a prime element, we denote its prime ideal by $\pfr_a$ and its associated valuation by $v_a.$
	For $v\in \infty_K,$ let $\tau_v\colon K\hookrightarrow K_v$ be the embedding of $K$ into its completion. For $v\in \Omega_K^f,$ let $\Ocal_{K_v}$ be its valuation ring, and let $\FF_v$ be its residue field. %and $\pfr_v$ be the prime ideal associated to $v.$  
	Let $K^2$ denote the set of square elements of $K.$ Let $\Ocal_S=\bigcap_{v\in \Omega^f_K\backslash S}(K\cap \Ocal_{K_v})$ be the ring of $S$-integers. A strong approximation theorem \cite[Chapter II \S 15]{CF67} states that $K$ is dense in $\AA_K^S$ for any nonempty $S.$  %i.e. the affine line $\AA^1$ satisfies strong approximation off any nonempty $S.$ For the definition of strong approximation, one can refer to \cite{Co12} or \cite{LX15}). 
	In this paper, we only use the following special case:
	\begin{lemma}\label{lemma strong approximation for A^1}
		The set $K$ is dense in $\AA_K^{2_K}.$ %for $S=\{$all $2$-adic places of $K\}.$
	\end{lemma}

	It is not difficult to generalize \cite[Theorem 13.4]{Ne99} to the following version of \v{C}ebotarev's density theorem.
	\begin{theorem}[\v{C}ebotarev]\label{theorem Chebotarev density theorem}
		%Let $L$ be a finite extension of a number field $K.$ Then t
		The set of places of $K$ splitting completely in $L,$ has positive density. 
	\end{theorem}

	\subsection{Hilbert symbol}

	We use the Hilbert symbol $(a,b)_v\in \{\pm 1\},$ for $a,b\in K_v^\times$ and $v\in \Omega_K.$ By definition, $(a,b)_v=1$ if and only if $x_0^2-ax_1^2-bx_2^2=0$ has a $K_v$-solution in $\PP^2$ with homogeneous coordinates $(x_0:x_1:x_2),$ which equivalently means that the curve defined over $K_v$ by the equation $x_0^2-ax_1^2-bx_2^2=0$ in $\PP^2,$ is isomorphic to $\PP^1.$ The Hilbert symbol gives a symmetric bilinear form on $K_v^\times/K_v^{\times 2}$ with value in $\ZZ/2\ZZ,$ cf. \cite[Chapter XIV, Proposition 7]{Se79}. And this bilinear form is nondegenerate, cf. \cite[Chapter XIV, Corollary 7]{Se79}. 
	
	\subsection{Preparation lemmas} We state the following lemmas for later use.

	\begin{lemma}\label{lemma hilbert symbal lifting for odd prime}
		Let $v$ be an odd place of $K.$ Let $a, b\in K_v^\times$ such that $v(a),v(b)$ are even. Then $(a,b)_v=1.$
	\end{lemma}
	
	\begin{proof}
		Choose a prime element $\pi_v\in K_v.$ Let $a_1=a\pi_v^{-v(a)}$ and $b_1=b\pi_v^{-v(b)}.$ Since the valuations $v(a)$ and $v(b)$ are even, the elements $\pi_v^{-v(a)}$ and $\pi_v^{-v(b)}$ are in $K_v^{\times 2}.$ So $(a,b)_v=(a_1,b_1)_v$ and $a_1,b_1\in \Ocal_{K_v}^\times.$
		By Chevalley-Warning theorem (cf. \cite[Chapter I \S 2, Corollary 2]{Se73}), the equation $x_0^2-\bar{a}_1x_1^2-\bar{b}_1x_2^2=0$ has a nontrivial solution in $\FF_v.$ For $v$ is odd, by Hensel's lemma, this solution can be lifted to a nontrivial solution in $\Ocal_{K_v}.$ Hence $(a,b)_v=(a_1,b_1)_v=1.$
	\end{proof}
	
	\begin{lemma}\label{lemma Hensel lemma for Hilbert symbal}
		Let $v$ be an odd place of $K.$ Let $a,b,c\in K_v^\times$ such that $v(b)<v(c).$ Then $(a,b+c)_v=(a,b)_v.$
	\end{lemma}
	
	\begin{proof}
		Since $v(b)<v(c),$ we have $v(b^{-1}c)>0.$ By Hensel's lemma, we have $1+b^{-1}c\in K_v^{\times 2}.$ So $(a,b+c)_v=(a,b(1+b^{-1}c))_v=(a,b)_v.$
	\end{proof}

	The following two lemmas are well known. We omit their proofs.
	
	\begin{lemma}\label{lemma openness for K^2 and O_K cross}
		The set $K_v^{\times 2}$ is an open subgroup of $K_v^\times.$ If $v\in \Omega_K^f,$ then $\Ocal_{K_v}^\times$ is also an open subgroup of $K_v^\times.$ So, they are nonempty open subset of $K_v.$ 
	\end{lemma}

	\begin{lemma}\label{lemma openness for v(x)=n}
		Let $v\in \Omega_K^f.$ For any $n\in \ZZ,$ the set $\{x\in K_v|v(x)=n\}$ is a nonempty open subset of $K_v.$
	\end{lemma}

	\begin{lemma}\label{lemma openness for hilbert symbal 1}
		Let $v\in \Omega_K^f.$ For any $a\in K_v^\times,$ the sets $\{x\in K_v^\times|(a,x)_v=1\},$ $\{x\in K_v^\times|(a,x)_v=1\}\cap \Ocal_{K_v}$ and $\{x\in \Ocal_{K_v}^\times|(a,x)_v=1\}$ are nonempty open subsets of $K_v.$ 
	\end{lemma}
	
	\begin{proof}
		For the unit $1$ belongs to these sets, they are nonempty. By Lemma \ref{lemma openness for K^2 and O_K cross}, the sets $K_v^{\times 2}$ and $\Ocal_{K_v}^\times$ are nonempty open subsets of $K_v.$ The set $\{x\in K_v^\times|(a,x)_v=1\}$ is a union of cosets of $K_v^{\times 2}$ in the group $K_v^\times.$  So the sets are open in $K_v.$
	\end{proof}

	\begin{lemma}\label{lemma openness for hilbert symbal -1}
		Let $v\in \Omega_K^f.$ For any $a\in K_v^\times,$ the sets $\{x\in K_v^\times|(a,x)_v=-1\}$ and  $\{x\in K_v^\times|(a,x)_v=-1\}\cap \Ocal_{K_v}$ are open subsets of $K_v.$ Furthermore, if $a\notin K_v^{\times 2},$ then they are nonempty. 
	\end{lemma}
	
	\begin{proof}
		If the set $\{x\in K_v^\times|(a,x)_v=-1\}\neq \emptyset,$ then it is a union of cosets of $K_v^{\times 2}$ in the group $K_v^\times.$  By Lemma \ref{lemma openness for K^2 and O_K cross}, it is an open subset of $K_v.$ For $\Ocal_{K_v}$ is open in $K_v,$ the sets are open subsets of $K_v^\times.$  Nonemptiness is from the nondegeneracy of the bilinear form given by the Hilbert symbol, and from multiplying a square element in $\Ocal_{K_v}$ to denominate an element in $K_v^\times.$
	\end{proof}

	\begin{lemma}\label{lemma openness for hilbert symbal with odd valuation element}
		Let $v\in \Omega_K^f.$ For any $a\in K_v^\times$ with $v(a)$ odd, the set $\{x\in \Ocal_{K_v}^\times|(a,x)_v=-1\}$ is a nonempty open subset of $K_v.$
	\end{lemma}
	
	\begin{proof}
		By Lemmas \ref{lemma openness for K^2 and O_K cross} and \ref{lemma openness for hilbert symbal -1}, the set is open in $K_v.$ We need to show that it is nonempty. For $a\notin K_v^2,$ by the nondegeneracy of the bilinear form given by the Hilbert symbol, there exists an element $b\in K_v^\times$ such that $(a,b)_v=-1.$ If $v(b)$ is odd, let $b'=-ab.$ Then $(a,b')_v=(a,-ab)_v=(a,-a)_v(a,b)_v=-1.$ Replacing $b$ by $b'$ if necessary, we can assume that $v(b)$ is even. Choose a prime element $\pi_v\in K_v.$ Then $\pi_v^{-v(b)}\in K_v^{\times 2},$ so the element $b\pi_v^{-v(b)}$ is in this set.
	\end{proof}

	\subsection{Brauer-Manin obstruction}
	Cohomological obstructions have been used to explain failures of the Hasse principle and nondensity of $X(K)$ in $X(\AA_K^S).$  Let $\Br(X)=H^2_{\et}(X,\GG_m)$ be the Brauer group of $X.$ Let $\inv_v\colon \Br(K_v)\to \QQ/\ZZ$ be the local invariant map. The Brauer-Manin pairing
	$$X(\AA_K)\times\Br(X)\to \QQ/\ZZ,$$ 
	suggested by Manin \cite{Ma71}, between $X(\AA_K)$ and $\Br(X),$ is provided by local class field theory. The left kernel of this pairing is denoted by $X(\AA_K)^{\Br},$ which is a closed subset of $X(\AA_K).$ By the global reciprocity in class field theory, there is an exact sequence: 
	$$0\to \Br(K)\to \bigoplus\limits_{v\in \Omega_K} \Br(K_v)\to \QQ/\ZZ\to 0,$$
	which induces an inclusion: $X(K)\subset pr^S(X(\AA_K)^\Br).$ 
	\begin{remark}
			For any smooth, proper and rationally connected variety $X$ defined over a number field $K,$ it is conjectured by Colliot-Th\'el\`ene \cite{CT03} that the $K$-rational points set $X(K)$ is dense in  $X(\AA_K)^\Br.$  Colliot-Th\'el\`ene's conjecture holds for Ch\^atelet surfaces, cf. \cite{CTSSD87a,CTSSD87b}.
	\end{remark}

	\section{Ch\^atelet surfaces}\label{section: Ch\^atelet surfaces and choose an element a for S}

	Let $K$ be a number field. Given an equation (\ref{equation}), let $V^0$ be the affine surface in $\AA^3_K$ defined by this equation. Let $V$ be the natural smooth compactification  of $V^0$ given in \cite[Section 7.1]{Sk01}, which is called the Ch\^atelet surface given by this equation, cf. \cite[Section 5]{Po09}. Notice that all smooth projective models of a given equation (\ref{equation}) are the same as to the discussion of the Hasse principle and weak approximation.
	
	\begin{remark}\label{remark birational to PP^2}
		For any local field $K_v,$  if $a\in K_v^{\times 2},$ then $V$ is birationally equivalent to $\PP^2$ over $K_v.$ By the implicit function theorem, there exists a $K_v$-point on $V.$
	\end{remark}
	
	\begin{remark}\label{remark the implicit function thm and local constant Brauer group}
		For any local field $K_v,$ by smoothness of $V,$ the implicit function theorem implies that the nonemptiness of $V^0(K_v)$ is equivalent to the nonemptiness of $V(K_v),$ and that $V^0(K_v)$ is open dense in $V(K_v)$ with the $v$-adic topology. Given an element $A\in\Br(V),$ the evaluation of $A$ on $V(K_v)$ is locally constant. By the properness of $V,$ the space $V(K_v)$ is compact. So the set of all possible values of the evaluation of $A$ on $V(K_v)$ is finite. Indeed, by \cite[Proposition 7.1.2]{Sk01}, there exist only two possible values. They are determined by the evaluation of $A$ on $V^0(K_v).$ In particular, if the evaluation of $A$ on $V^0(K_v)$ is constant, then it is constant on $V(K_v).$ 
	\end{remark}

		In the following two sections, we will construct two kinds of Ch\^atelet surfaces. 
		
			Given an extension of number fields $L/K,$
		and a finite subset $S\subset \Omega_K\backslash (\infty_K^c\cup 2_K),$
	we always use the following way to choose an element $a\in \Ocal_K\backslash K^2$ for the parameter $a$ in the equation (\ref{equation}).

	If $S=\emptyset,$ by Theorem \ref{theorem Chebotarev density theorem}, we can take a place $v_0\in \Omega_K^f\backslash 2_K$ splitting completely in $L.$ Then replace $S$ by $\{v_0\}$ to continue the following step.
	
	Now, suppose that $S\neq \emptyset.$ For $v\in \Omega_K,$ by Lemma \ref{lemma openness for K^2 and O_K cross}, the set $K_v^{\times 2}$ is a nonempty open subset of $K_v.$ For $v\in \Omega_K^f,$ by Lemma \ref{lemma openness for v(x)=n}, the set  $\{a\in K_v| v(a)$ is odd$\}$ is a nonempty open subset of $K_v.$  Using weak approximation for the affine line $\AA^1,$ we can choose an element $a\in K^\times$ satisfying the following conditions:
	\begin{itemize}
		\item $\tau_v(a)<0$ for all $v\in S\cap \infty_K,$
		%	\item $\tau_v(a)>0$ for all real places $v\in \infty_K\backslash S,$
		\item $a\in K_v^{\times 2}$  for all $v\in 2_K,$
		\item $v(a)$ is odd for all $v\in S\backslash \infty_K.$
	\end{itemize}
	These conditions do not change by multiplying an element in $K^{\times 2},$ so we can assume $a\in \Ocal_K.$ The conditions that $v(a)$ is odd for all $v\in S\backslash \infty_K,$ and that $\tau_v(a)<0$ for all $v\in S\cap \infty_K,$ imply $a\in \Ocal_K\backslash K_v^2$ for all $v\in S.$ So $a\in \Ocal_K\backslash K^2.$ 
	
	\begin{remark}\label{remark choose an element a for S remark 1}
		Let $S'=\{v\in \infty_K^r | \tau_v(a)<0\}\cup \{v\in \Omega_K^f\backslash 2_K | v(a){\rm ~is ~odd} \},$ then $S'$ is a finite set. By the conditions that $\tau_v(a)<0$ for all $v\in S\cap \infty_K,$ and that $v(a)$ is odd for all $v\in S\backslash \infty_K,$ we have $S'\supset S.$ Then $S'\neq \emptyset.$
	\end{remark}
	
	\begin{remark}\label{remark choose an element a for S 2}
		If there exists one place in $S$ splitting completely in $L$ or $S=\emptyset,$ then by the choice of $a,$ the element $a\in \Ocal_K\backslash L^2.$		
	\end{remark}

	\section{Weak approximation under extensions of base fields}

	In the paper \cite{Li18}, Liang  study the non-invariance of weak approximation under an extension of base fields. More precisely, for any number field $K,$ Liang \cite[Proposition 3.4]{Li18} proved that there exist a Ch\^atelet surface $V$ over $K$ and a quadratic
	extension $L/K$	 such that $V(K)\neq \emptyset,$ $V$ satisfies weak approximation, but the
	base extension $V_L$ does not satisfies weak approximation, even off $\infty_L.$ In this section, we generalize Liang's construction and obtain	further results that apply to an arbitrary extension $L/K.$

	\subsection{Choice of parameters for the equation (\ref{equation})}\label{subsection Choice of parameters for V2}
	By choosing $a\in \Ocal_K\backslash K^2$ as in Section \ref{section: Ch\^atelet surfaces and choose an element a for S}, we  choose an element $b\in K^\times$ in the following way.
	
	Let $S'=\{v\in \infty_K^r | \tau_v(a)<0\}\cup \{v\in \Omega_K^f\backslash 2_K | v(a){\rm ~is ~odd} \}$ be as in Remark \ref{remark choose an element a for S remark 1}, then $S'\supset S$ is a finite set. 
	By Lemma \ref{lemma openness for v(x)=n}, for $v\in S\backslash \infty_K,$ the set $\{b\in K_v|v(b)=-v(a)\}$ is a nonempty open subset of $K_v;$ for $v\in S'\backslash (S\cup \infty_K),$ the set $\{b\in K_v|v(b)=v(a)\}$ is a nonempty open subset of $\Ocal_{K_v}.$ By Lemma \ref{lemma strong approximation for A^1}, we can choose a nonzero element $b\in \Ocal_S[1/2]$ satisfying the following conditions:
	\begin{itemize}
		%\item $\tau_v(b)<0$ for all $v\in S\cap \infty_K,$
		\item $v(b)=-v(a)$ for all $v\in S\backslash \infty_K,$
		\item $v(b)=v(a)$  for all $v\in S'\backslash (S\cup \infty_K).$
		%\item $v_0(c)=1$ for the chosen $v_0$ above.
	\end{itemize}	
	
	We  choose an element $c\in K^\times$ with respect to the chosen $a,b$ in the following way.

	Let $S''=\{v\in \Omega_K^f\backslash 2_K| v(b)\neq 0 \},$ then $S''$ is a finite set and $S'\backslash \infty_K \subset S''.$ 
	By Theorem \ref{theorem Chebotarev density theorem}, we can take two different finite places $v_1,v_2\in \Omega_K^f\backslash  S''$ splitting completely in $L.$ 	
	If $v \in S\backslash \infty_K,$ then $v(a)$ is odd. In this case, by Lemma \ref{lemma openness for hilbert symbal with odd valuation element}, the set  $\{c\in \Ocal_{K_v}^\times|(a,c)_v=-1\}$ is a nonempty open subset of $\Ocal_{K_v}.$ 
	If $v\in \{v_1,v_2\},$ then $b\in \Ocal_{K_v}^\times.$ In this case, by Lemma \ref{lemma openness for v(x)=n}, the sets $\{c\in K_v|v(c)=1\}$ and $\{c\in K_v|v(1+cb^2)=1\}$ are nonempty open subsets of $\Ocal_{K_v}.$ Also by Lemma \ref{lemma strong approximation for A^1}, we can choose a nonzero element $c\in\Ocal_K[1/2]$ satisfying the following conditions:
	\begin{itemize}
		\item $\tau_v(1+cb^2)<0$ for all $v\in S\cap \infty_K,$
		\item $\tau_v(c)>0$ for all $v\in (S'\backslash S)\cap \infty_K,$
		\item $(a,c)_v=-1$ and $v(c)=0$ for all $v\in S\backslash \infty_K,$
		%\item $v(b^3+c)>0$ for all $v\in S'\backslash S,$
		%\item $v(c)=1$ for all $v\in S''\backslash S',$
		\item $v_1(c)=1$ and $v_2(1+cb^2)=1$ for the chosen $v_1,v_2$ above.
	\end{itemize}
	
	Let $P(x)=(cx^2+1)((1+cb^2)x^2+b^2),$ and let $V_1$ be the Ch\^atelet surface given by $y^2-az^2=(cx^2+1)((1+cb^2)x^2+b^2).$

	\begin{proposition}\label{proposition the valuation of Brauer group on local points are fixed outside S and take two value on S}
		For any extension of number fields $L/K,$ and any finite subset $S \subset \Omega_K\backslash (\infty_K^c\cup 2_K)$ splitting completely in $L,$
		there exists a Ch\^atelet surface $V_1$ defined over $K,$ which has the following properties.
		\begin{itemize}
			\item The Brauer group $\Br(V_1)/\Br(K)\cong\Br(V_{1L})/\Br(L)\cong \ZZ/2\ZZ,$ is generated by an element $A\in \Br(V_1).$ The subset $V_1(K)\subset V_1(L)$ is nonempty. 	
			\item For any $v\in S,$ there exist $P_v$ and $Q_v$ in $V_1(K_v)$ such that the local invariants $\inv_v(A(P_v))=0$ and $\inv_v(A(Q_v))=\half.$  For any other $v\notin S,$ and any $P_v\in V_1(K_v),$ the local invariant $\invap=0.$	
			\item For any $v'\in S_L,$ there exist $P_{v'}$ and $Q_{v'}$ in $V_1(L_{v'})$ such that the local invariants $\inv_{v'}(A(P_{v'}))=0$ and $\inv_{v'}(A(Q_{v'}))=\half.$  For any other $v'\notin S_L,$ and any $P_{v'}\in V_1(L_{v'}),$ the local invariant $\inv_{v'}(A(P_{v'})) =0.$
		\end{itemize}
	\end{proposition}
	
	\begin{proof}
		For the extension $L/K,$ and the finite set $S,$ we  check that the Ch\^atelet surface $V_1$ chosen as in Subsection \ref{subsection Choice of parameters for V2}, has the properties.
			
		By the choice of the places $v_1,$ the polynomial $x^2+c$ is an Eisenstein polynomial, so it is irreducible over $K_{v_1}.$ Since $v_1(a)$ is even, we have $K(\sqrt{a})K_{v_1}\ncong K_{v_1}[x]/(cx^2+1).$ So $K(\sqrt{a})\ncong K[x]/(cx^2+1).$ The same argument holds for the place $v_2$ and polynomial $(1+cb^2)x^2+b^2.$	
		For all places of $S$ split completely in $L,$ then by Remark \ref{remark choose an element a for S 2}, we have $a\in \Ocal_K\backslash L^2.$ 
		By the splitting condition of $v_1, v_2,$ we have $L(\sqrt{a})\ncong L[x]/(cx^2+1)$ and $L(\sqrt{a})\ncong L[x]/((1+cb^2)x^2+b^2).$ So $P(x)=(cx^2+1)((1+cb^2)x^2+b^2)$ is separable and a product of two degree-2 irreducible factors over $K$ and $L.$ 
		According to \cite[Proposition 7.1.1]{Sk01}, the Brauer group $\Br(V_1)/\Br(K)\cong\Br(V_{1L})/\Br(L)\cong \ZZ/2\ZZ.$ Furthermore, by Proposition 7.1.2 in loc. cit, we take the quaternion algebra $A=(a,cx^2+1)\in \Br(V_1)$ as a generator element of this group. Then we have the equality $A=(a,cx^2+1)=(a,(1+cb^2)x^2+b^2)$ in $\Br(V_1).$ 
		
		For $(x,y,z)=(0,b,0)$ is a rational point on $V_1^0,$ the set $V_1(K)$ is nonempty. We denote this rational point by $Q_0.$
		
		We need to compute the evaluation of $A$ on $V_1(K_v)$ for all $v\in \Omega_K.$ 
		
		For any $v\in \Omega_K,$ the local invariant $\invaqo=0.$ By Remark \ref{remark the implicit function thm and local constant Brauer group}, it suffices to compute the local invariant $\invap$ for all $P_v\in V_1^0(K_v).$
		
		\begin{enumerate}
		\item Suppose that $v\in (\infty_K\backslash S')\cup 2_K.$  Then $a\in K_v^{\times 2},$ so $\invap=0$ for all $P_v\in V_1(K_v).$
		\item Suppose that $v\in (S'\backslash S)\cap \infty_K.$ For any $x\in K,$ by the choice of $c,$ we have $\tau_v(cx^2+1)>0.$ Then $(a,cx^2+1)_v=1,$ so $\invap=0$ for all $P_v\in V_1(K_v).$
		\item Suppose that $v\in S'\backslash (S\cup \infty_K).$ Take an arbitrary $P_v\in V_1^0(K_v).$ If $\invap=1/2,$ then $(a,cx^2+1)_v=-1=(a,(1+cb^2)x^2+b^2)_v$ at $P_v.$ By Lemma \ref{lemma Hensel lemma for Hilbert symbal}, the first equality implies $v(x)\leq 0.$ Since $v(a)=v(b)>0$ and $v(c)\geq 0,$ by Lemma \ref{lemma Hensel lemma for Hilbert symbal}, we have  $(a,(1+cb^2)x^2+b^2)_v=(a,x^2)_v=1,$ which is a contradiction. So $\invap=0.$
		\item Suppose that $v\in \Omega_K^f\backslash (S'\cup 2_K ).$ Take an arbitrary $P_v\in V_1^0(K_v).$ If $\invap=1/2,$ then $(a,cx^2+1)_v=-1=(a,(1+cb^2)x^2+b^2)_v$ at $P_v.$ Since $v(a)$ is even, by Lemma \ref{lemma hilbert symbal lifting for odd prime}, the first equality implies that $v(cx^2+1)$ is odd. Since $c\in\Ocal_K[1/2],$ we have $v(x)\leq 0.$ So $v(c+x^{-2})$ is odd and positive. Since $v(b)\geq 0,$ by Hensel's lemma, we have $1+b^2(c+x^{-2})\in K_v^{\times 2}.$ So $(a,(1+cb^2)x^2+b^2)_v=(a,x^2)_v(a,1+b^2(c+x^{-2}))_v=1,$ which is a contradiction. So $\invap=0.$		
		\item Suppose that $v\in S\cap\infty_K.$ Take $P_v=Q_0,$ then $\invap=0.$ By the choice of $b,c,$ we have $\tau_v(\frac{b^2}{-cb^2-1})>\tau_v(\frac{1}{-c})>0.$ Take $x_0\in K$ such that $\tau_v(x_0)>\sqrt{\tau_v(\frac{b^2}{-cb^2-1})},$ then $\tau_v((cx_0^2+1)((1+cb^2)x_0^2+b^2))>0$ and $\tau_v(cx_0^2+1)<0.$ So there exists a $Q_v\in V_1^0(K_v)$ with $x=x_0.$ Then $\invaq=\half.$
		\item Suppose that $v\in S\backslash\infty_K.$ Take $P_v=Q_0,$ then $\invap=0.$ Take $x_0\in K_v$ such that  $v(x_0)< 0.$ Since $v(b)=-v(a)<0$ and $v(c)=0,$ by Lemma \ref{lemma Hensel lemma for Hilbert symbal}, we have $(a,cx_0^2+1)_v=(a,cx_0^2)_v=(a,c)_v$ and $(a,(1+cb^2)x_0^2+b^2)_v=(a,cb^2x_0^2)_v=(a,c)_v.$ So
		$(a,(cx_0^2+1)((1+cb^2)x_0^2+b^2))_v=(a,c)_v(a,c)_v=1.$ Hence, there exists a $Q_v\in V_1^0(K_v)$ with $x=x_0.$ Since $(a,c)_v=-1,$ we have $\invaq=\half.$ 	
	\end{enumerate}

		Finally, we need to compute the evaluation of $A$ on $V_1(L_{v'})$ for all $v'\in \Omega_L.$
		
		For any $v'\in \Omega_L,$ the local invariant $\inv_{v'}(A(Q_0))=0.$
		\begin{enumerate}
		\item Suppose that $v' \in S_L.$ Let $v\in \Omega_K$ be the restriction of $v'$ on $K.$ By the assumption that $v$ splits completely in $L,$ we have $K_v=L_{v'}.$ So $V_1(K_v)=V_1(L_{v'}).$	By the argument already shown, there exist $P_v, Q_v\in V_1(K_v)$ such that $\inv_v(A(P_v))=0$ and $\inv_v(A(Q_v))=\half.$ View $P_v, Q_v$ as elements in $V_1(L_{v'}),$ and let $P_{v'}=P_v$ and $Q_{v'}=Q_v.$ Then $\inv_{v'}(A(P_{v'}))=\invap=0$ and $\inv_{v'}(A(Q_{v'}))=\invaq=\half.$ 
		\item Suppose that $v'\in \Omega_L\backslash S_L.$ This local computation is the same as the case $v\in \Omega_K\backslash S.$	
	\end{enumerate}
	\end{proof}

	\begin{remark}\label{remark union of the local invariant 0 and one-half}
		For any $v\in S,$ and any $P_v\in V_1(K_v),$  the local invariant of the evaluation of $A$ on $P_v$ is $0$ or $\half.$ Let $U_1=\{P_v\in V_1(K_v)| \invap=0\}$ and $U_2=\{P_v\in V_1(K_v)| \invap=\half\}.$ Then $U_1$ and $U_2$ are nonempty disjoint open subsets of $V_1(K_v),$ and $V_1(K_v)=U_1\bigsqcup U_2.$
	\end{remark}

	Applying the global reciprocity law, the surface $V_1$ in Proposition \ref{proposition the valuation of Brauer group on local points are fixed outside S and take two value on S}, has the following weak approximation properties.
	
	\begin{proposition}\label{proposition the valuation of Brauer group on local points are fixed outside S and take two value on S property}
		Given an extension of number fields $L/K,$
		and a finite subset $S \subset \Omega_K\backslash (\infty_K^c\cup 2_K)$ splitting completely in $L,$ let $V_1$ be a Ch\^atelet surface satisfying those properties of Proposition \ref{proposition the valuation of Brauer group on local points are fixed outside S and take two value on S}.
			\begin{enumerate}
		 \item If $S=\emptyset,$ then $V_1$ and $V_{1L}$ satisfy weak approximation.  
		 \item If $S\neq \emptyset,$ then $V_1$ satisfies weak approximation off $S'$ for a finite subset $S'\subset \Omega_K$ if and only if  $S'\cap S\neq \emptyset.$
		 \item  If $S\neq \emptyset,$ the surface $V_{1L}$ satisfies weak approximation off $T$ for a finite subset $T\subset \Omega_L$ if and only if $T\cap S_L\neq \emptyset.$ 
		 \end{enumerate}
	\end{proposition}
	
	\begin{proof}
		
		According to \cite[Theorem B]{CTSSD87a,CTSSD87b}, the Brauer-Manin obstruction to the Hasse principle and weak approximation is the only one for Ch\^atelet surfaces, so $V_1(K)$ is dense in $V_1(\AA_K)^{\Br}.$
			\begin{enumerate}
		 \item Suppose that $S=\emptyset,$ then for any $(P_v)_{v\in\Omega_K}\in V_1(\AA_K),$ by Proposition \ref{proposition the valuation of Brauer group on local points are fixed outside S and take two value on S}, the sum $\sum_{v\in \Omega_K}\inv_v(A(P_v))=0.$ Since $\Br(V_1)/\Br(K)$ is generated by the element $A,$ we have $V_1(\AA_K)^{\Br}= V_1(\AA_K).$ So $V_1(K)$ is dense in $V_1(\AA_K)^{\Br}= V_1(\AA_K),$ i.e. the surface $V_1$ satisfies weak approximation.  
		\item \begin{enumerate}
			 \item	Suppose that $S'\cap S\neq \emptyset.$ Take $v_0\in S'\cap S.$ 
			For any finite subset $R\subset \Omega_K\backslash\{v_0\},$ take a nonempty open subset $M=V_1(K_{v_0})\times \prod_{v\in R}U_v\times \prod_{v\notin R \cup\{v_0\}}V_1(K_v)\subset V_1(\AA_K).$ Take an element $(P_v)_{v\in \Omega_K}\in M$ with $\inv_{v_0} A(P_{v_0})=0.$	 By Proposition \ref{proposition the valuation of Brauer group on local points are fixed outside S and take two value on S} and $v_0\in S,$ we can take an element $P_{v_0}'\in V_1(K_{v_0})$ such that $\inv_{v_0} A(P_{v_0}')=\half.$	
			By Proposition \ref{proposition the valuation of Brauer group on local points are fixed outside S and take two value on S}, the sum $\sum_{v\in \Omega_K\backslash \{v_0\}}\inv_v(A(P_v))$ is $0$ or $\half$ in $\QQ/\ZZ.$ If it is $\half,$ then we replace $P_{v_0}$ by $P_{v_0}'.$ In this way, we get a new element $(P_v)_{v\in \Omega_K}\in M.$  And the sum $\sum_{v\in \Omega_K}\inv_v(A(P_v))=0$ in $\QQ/\ZZ.$ So $(P_v)_{v\in \Omega_K}\in V_1(\AA_K)^{\Br}\cap M.$ Since $V_1(K)$ is dense in $V_1(\AA_K)^{\Br},$ the set $V_1(K)\cap M\neq \emptyset,$ which implies that $V_1$ satisfies weak approximation off $\{v_0\}.$ So $V_1$ satisfies weak approximation off $S'.$
			
		 \item	Suppose that $S\neq \emptyset$ and $S'\cap S= \emptyset.$ Take $v_0\in S,$ and let $U_{v_0}=\{P_{v_0}\in V_1(K_{v_0})| \inv_{v_0}(A(P_{v_0}))=\half\}.$ For $v\in S\backslash \{v_0\},$  let $U_v=\{P_v\in V_1(K_v)| \invap=0\}.$ For any $v\in S,$ by Remark \ref{remark union of the local invariant 0 and one-half}, the set $U_v$ is a nonempty  open subset of $V_1(K_v).$ Let $M=\prod_{v\in S}U_v\times \prod_{v\notin S }V_1(K_v).$ It is a nonempty  open subset of $V_1(\AA_K).$ For any $(P_v)_{v\in \Omega_K}\in M,$ by Proposition \ref{proposition the valuation of Brauer group on local points are fixed outside S and take two value on S} and the choice of $U_v,$ the sum $\sum_{v\in \Omega_K}\inv_v(A(P_v))=\half$ is nonzero in $\QQ/\ZZ.$ So $V_1(\AA_K)^{\Br}\cap M=\emptyset,$ which implies  $V_1(K)\cap M= \emptyset.$ Hence $V_1$ does not satisfy weak approximation off $S'.$
		\end{enumerate}
	 \item The same argument applies to $V_{1L}.$ 		
	\end{enumerate}
		
	\end{proof}

	From discussion of Proposition \ref{proposition the valuation of Brauer group on local points are fixed outside S and take two value on S property}, we have the following weak approximation properties for Ch\^atelet surfaces.

	\begin{thm}\label{thm: interesting result for Chatelet surface1}
		For any extension of number fields $L/K,$ and any finite nonempty subset $S \subset \Omega_K\backslash (\infty_K^c\cup 2_K)$ splitting completely in $L,$
		there exists a Ch\^atelet surface $V$ over $K$ with $V(K)\neq \emptyset,$ such that for every intermediate field $K\subset L'\subset L,$
		the Brauer group $\Br(V)/\Br(K)\cong\Br(V_{L'})/\Br(L')\cong \ZZ/2\ZZ.$ For every finite subset $T'\subset \Omega_{L'},$ the base extension $V_{L'}$ satisfies weak approximation off $T'$ if and only if  $T'\cap S_{L'}\neq \emptyset.$ In particular,  the surface $V_{L'}$ does not satisfy weak approximation for every $K\subset L'\subset L.$
	\end{thm}	
	
	\begin{proof}
		For the extension $L/K$ and the set $S,$ let $V$ be the Ch\^atelet surface chosen as in Subsection \ref{subsection Choice of parameters for V2}.  Applying the same argument about the field $L$ to its subfield $L',$ the properties that we list, are just what we have explained in Proposition \ref{proposition the valuation of Brauer group on local points are fixed outside S and take two value on S} and Proposition \ref{proposition the valuation of Brauer group on local points are fixed outside S and take two value on S property}.
	\end{proof}

	Using the construction method in Subsection \ref{subsection Choice of parameters for V2}, we have the following example, which is a special case of Proposition \ref{proposition the valuation of Brauer group on local points are fixed outside S and take two value on S}. 
	
	\begin{eg}\label{example1: construction of V_0}
		For $K=\QQ$ and $L=\QQ(\sqrt{3}),$ and let $S=\{73\}\subset\Omega_K.$ The prime numbers $11,23,73$ split completely in $L.$ Using the construction method in Subsection \ref{subsection Choice of parameters for V2}, we choose data: $S=S'=S''={73},~v_1=11,~v_2=23,~a=73,~b=1/73,~c=99$ and  $P(x)=(99x^2+1)(5428x^2/5329+1/5329).$ Then the Ch\^atelet surface given by $y^2-73z^2=P(x),$ has the properties of  Propositions \ref{proposition the valuation of Brauer group on local points are fixed outside S and take two value on S} and \ref{proposition the valuation of Brauer group on local points are fixed outside S and take two value on S property}.
	\end{eg}
	%\begin{remark}
	%	Let $S_1$ and $S_2$ be two finite sets in $\Omega_K$ and both $S_1$ and $S_2$ satisfy the conditions of Proposition \ref{corollary the valuation of Brauer group on local points are fixed and on S_L nontrivial}, then there exist $a\in K\backslash L^2$ and relatively prime separable degree-4 polynomials $P_1$ and $P_2$ such that the Ch\^atelet surfaces given by $y^2-az^2=P_1$ and $y^2-az^2=P_1$ $W_1$ and $W_2$ 
	%\end{remark}

	\section{The Hasse principle under extensions of base fields}

	Iskovskikh \cite{Is71} gave an example of the intersection of two quadratic hypersurfaces in $\PP^4_\QQ,$ which is a Ch\^atelet surface over
	$\QQ$ given by $y^2+z^2=(x^2-2)(-x^2+3).$ He showed that this Ch\^atelet surface is a counterexample to the Hasse principle. Similarly, Skorobogatov \cite[Pages 145-146]{Sk01} gave a family of Ch\^atelet surfaces with a parameter over $\QQ.$ He discussed the property of the Hasse principle for this family. Poonen \cite[Proposition 5.1]{Po09} generalized their arguments to any number field. Given a number field $K,$ he constructed a Ch\^atelet surface defined over $K,$ which is a counterexample to the Hasse principle. He used \v{C}ebotarev's density theorem for some ray class fields to choose the parameters for the equation (\ref{equation}). The Ch\^atelet surface that he constructed, has the properties of \cite[Lemma 5.5]{Po09} (a special situation of the following Proposition \ref{proposition the valuation of Brauer group on local points are fixed and on S nontrivial}: the case when $S=\{v_0\}$ for some place $v_0$ associated to some large prime element in $\Ocal_K$), which is the main ingredient in the proof of \cite[Proposition 5.1]{Po09}. In this section, we generalize them and consider the Hasse principle of Ch\^atelet surfaces under extensions of base fields.

	\subsection{Choice of parameters for the equation (\ref{equation})}\label{subsection Choice of parameters for V3}
	By choosing $a\in \Ocal_K\backslash K^2$ as in Section \ref{section: Ch\^atelet surfaces and choose an element a for S}, we  choose an element $b\in K^\times$ in the following way.
	
	Let $S'=\{v\in \infty_K^r | \tau_v(a)<0\}\cup \{v\in \Omega_K^f\backslash 2_K | v(a){\rm ~is ~odd} \}$ be as in Remark \ref{remark choose an element a for S remark 1}, then $S'\supset S$ is a finite set. If $v \in S\backslash \infty_K,$ then $v(a)$ is odd. Then by Lemma \ref{lemma openness for hilbert symbal with odd valuation element}, the set  $\{b\in \Ocal_{K_v}^\times|(a,b)_v=-1\}$ is a nonempty open subset of $\Ocal_{K_v}.$ If $v \in S'\backslash (S\cup \infty_K),$ then by Lemma \ref{lemma openness for hilbert symbal 1}, the set  $\{b\in \Ocal_{K_v}^\times|(a,b)_v=1\}$ is a nonempty open subset of $\Ocal_{K_v}.$ By Lemma \ref{lemma strong approximation for A^1}, we can choose a nonzero element $b\in \Ocal_K[1/2]$ satisfying the following conditions:
	\begin{itemize}
		\item $\tau_v(b)<0$ for all $v\in S\cap \infty_K,$
		\item $\tau_v(b)>0$ for all $v\in (S'\backslash S)\cap \infty_K,$
		\item $(a,b)_v=-1$ and $v(b)=0$ for all $v\in S\backslash \infty_K,$
		\item $(a,b)_v=1$ and $v(b)=0$  for all $v\in S'\backslash (S\cup \infty_K).$
	\end{itemize}	
	%	Then $a,b$ are coprime in $\Ocal_K[1/2]$ and $b\in \Ocal_{K_v}^\times$ for all $v\in S'.$
	
	We  choose an element $c\in K^\times$ with respect to the chosen $a,b$ in the following way.
	
	Let $S''=\{v\in \Omega_K^f\backslash  2_K | v(b)\neq 0 \},$ then $S''$ is a finite set and $S'\cap S''=\emptyset.$ By Theorem \ref{theorem Chebotarev density theorem}, we can take two different finite places $v_1,v_2\in \Omega_K^f\backslash (S'\cup S''\cup 2_K )$ splitting completely in $L.$ If $v\in (S'\backslash \infty_K)\cup \{v_1,v_2\},$ then $b\in\Ocal_{K_v}^\times.$ In this case, by Lemma \ref{lemma openness for v(x)=n}, the sets $\{c\in K_v|v(bc+1)=v(a)+2\},$ $\{c\in K_v|v(c)=1\}$ and $\{c\in K_v|v(bc+1)=1\}$ are nonempty open subsets of $\Ocal_{K_v}.$
	%the set $U=\{x\in K_v|v(x)=v(a)+2\}$ is a nonempty open subset of $\Ocal_{K_v}.$ Then, for $b\in\Ocal_{K_v}^\times,$
	If $v\in S'',$ by Lemma \ref{lemma openness for hilbert symbal 1}, the set $\{c\in \Ocal_{K_v}^\times|(a,c)_v=1\}$ is a nonempty open subset of $\Ocal_{K_v}.$
	Also by Lemma \ref{lemma strong approximation for A^1}, we can choose a nonzero element $c\in\Ocal_K[1/2]$ satisfying the following conditions:
	\begin{itemize}
		\item $0<\tau_v(c)<-1/\tau_v(b)$ for all $v\in S\cap \infty_K,$
		\item $\tau_v(bc+1)<0$ for all $v\in (S'\backslash S)\cap \infty_K,$
		\item $v(bc+1)=v(a)+2$ for all $v\in S'\backslash \infty_K,$
		\item $(a,c)_v=1$ for all $v\in S'',$
		\item $v_1(c)=1$ and $v_2(bc+1)=1$ for the chosen $v_1,v_2$ above.
	\end{itemize}

	Let $P(x)=(x^2-c)(bx^2-bc-1),$ and let $V_2$ be the Ch\^atelet surface given by $y^2-az^2=(x^2-c)(bx^2-bc-1).$ 
	
	\begin{proposition}\label{proposition the valuation of Brauer group on local points are fixed and on S nontrivial}
		For any extension of number fields $L/K,$ and any finite subset $S \subset \Omega_K\backslash (\infty_K^c\cup 2_K)$ splitting completely in $L,$ there exists a Ch\^atelet surface $V_2$ defined over $K,$ which has the following properties.
		\begin{itemize}		
			\item The Brauer group $\Br(V_2)/\Br(K)\cong\Br(V_{2L})/\Br(L)\cong \ZZ/2\ZZ,$ is generated by an element $A\in \Br(V_2).$ The subset $V_2(\AA_K)\subset V_2(\AA_L)$ is nonempty.	
			\item For any $v\in \Omega_K,$ and any $P_v\in V_2(K_v),$ 
			\begin{equation*}
				\inv_v(A(P_v))=\begin{cases}
					0& if\quad v\notin S,\\
					1/2 & if \quad v\in S.
				\end{cases}
			\end{equation*}
			\item For any $v'\in \Omega_L,$ and any $P_{v'}\in V_2(L_{v'}),$ 
			\begin{equation*}
				\inv_{v'}(A(P_{v'}))=\begin{cases}
					0& if\quad v'\notin S_L,\\
					1/2 & if \quad v'\in S_L.
				\end{cases}
			\end{equation*}
		\end{itemize}
	\end{proposition}
	
	\begin{proof}

		For the extension $L/K$ and the finite set $S,$ we  check that the Ch\^atelet surface $V_2$ chosen as in Subsection \ref{subsection Choice of parameters for V3}, has the properties.
		
		Firstly, we need to check that $V_2$ has an $\AA_K$-adelic point.
		\begin{enumerate}
		\item Suppose that $v\in (\infty_K\backslash S')\cup 2_K.$ Then $a\in K_v^{\times 2}.$ By Remark \ref{remark birational to PP^2}, the surface $V_2$ admits a $K_v$-point.
		\item Suppose that $v\in (S'\backslash S)\cap \infty_K.$ Let $x_0=0.$ Since $\tau_v(b)>0$ and $\tau_v(bc+1)<0,$ we have $\tau_v(c)<0$ and $\tau_v((x_0^2-c)(bx_0^2-bc-1))=\tau_v(c(bc+1))>0,$ which implies that $V_2^0$ admits a $K_v$-point with $x=0.$
		\item Suppose that $v\in S'\backslash (S\cup \infty_K).$  Take $x_0\in K_v$ such that the valuation $v(x_0)<0.$ Since $b\in \Ocal_{K_v}^\times$ and $c\in\Ocal_K[1/2],$ by Lemma \ref{lemma Hensel lemma for Hilbert symbal}, we have
		$(a,x_0^2-c)_v=(a,x_0^2)_v=1$ and $(a,bx_0^2-bc-1)_v=(a,bx_0^2)_v=(a,b)_v.$ By the choice of $b,$ we have $(a,b)_v=1.$ Hence $(a,(x_0^2-c)(bx_0^2-bc-1))_v=(a,b)_v=1,$ which implies that $V_2^0$ admits a $K_v$-point with $x=x_0.$
		\item Suppose that $v\in S''.$ By the choice of $a,b,c,$ we have $(a,c)_v=1,$ $bc+1\in \Ocal_{K_v}^\times,$ and the valuation $v(a)$ is even.  By Lemma \ref{lemma hilbert symbal lifting for odd prime}, we have $(a,bc+1)_v=1.$ Let $x_0=0.$ Then $(a,(x_0^2-c)(bx_0^2-bc-1))_v=(a,c(bc+1))_v=(a,c)_v(a,bc+1)_v=1,$ which implies that $V_2^0$ admits a $K_v$-point with $x=0.$
		\item Suppose that $v\in \Omega_K^f\backslash  (S'\cup S''\cup 2_K ).$ Then $v(b)=0.$  Take $x_0\in K_v$ such that the valuation $v(x_0)<0.$ Since $b\in \Ocal_{K_v}^\times$ and $c\in\Ocal_K[1/2],$ by Lemma \ref{lemma Hensel lemma for Hilbert symbal}, we have	
		$(a,x_0^2-c)_v=(a,x_0^2)_v=1$ and $(a,bx_0^2-bc-1)_v=(a,bx_0^2)_v=(a,b)_v.$ Since $v(a)$ and $v(b)$ are both even, by Lemma \ref{lemma hilbert symbal lifting for odd prime}, we have $(a,b)_v=1.$
		So $(a,(x_0^2-c)(bx_0^2-bc-1))_v=(a,b)_v=1,$ which implies that $V_2^0$ admits a $K_v$-point with $x=x_0.$
		\item Suppose that $v\in S\cap\infty_K.$ Let $x_0=0.$ Then by the choice of $a,b,c,$ we have $\tau_v(a)<0,$ $\tau_v(c)>0$ and $\tau_v(bc+1)>0.$ So $(a,(x_0^2-c)(bx_0^2-bc-1))_v=(a,c(bc+1))_v=1,$ which implies that $V_2^0$ admits a $K_v$-point with $x=0.$
		\item Suppose that $v\in S\backslash\infty_K.$ Choose a prime element $\pi_v$ and take $x_0=\pi_v.$ By the choice of $a,b,c,$ we have $b,c\in \Ocal_{K_v}^\times,~v(bx_0^2)=2,$ and $v(bc+1)=v(a)+2\geq 3.$ By Lemma \ref{lemma Hensel lemma for Hilbert symbal}, we have
		$(a,x_0^2-c)_v=(a,-c)_v$ and $(a,bx_0^2-bc-1)_v=(a,bx_0^2)_v.$ By Hensel's lemma, we have $-bc=1-(bc+1)\in K_v^{\times 2}.$
		So $(a,(x_0^2-c)(bx_0^2-bc-1))_v=(a,-bcx_0^2)_v=1,$ which implies that $V_2^0$ admits a $K_v$-point with $x=\pi_v.$		
	\end{enumerate}

		Secondly, we need to prove the statement about the Brauer group, and find the element $A$ in this proposition.
		
		By the choice of the places $v_1,$ the polynomial $x^2-c$ is an Eisenstein polynomial, so it is irreducible over $K_{v_1}.$ Since $v_1(a)$ is even, we have $K(\sqrt{a})K_{v_1}\ncong K_{v_1}[x]/(x^2-c).$ So $K(\sqrt{a})\ncong K[x]/(x^2-c).$ The same argument holds for the place $v_2$ and polynomial $bx^2-bc-1.$
		For all places of $S$ split completely in $L,$ then by Remark \ref{remark choose an element a for S 2}, we have $a\in \Ocal_K\backslash L^2.$ 
		By the splitting condition of $v_1, v_2,$ we have $L(\sqrt{a})\ncong L[x]/(x^2-c)$ and $L(\sqrt{a})\ncong L[x]/(bx^2-bc-1).$
		So $P(x)=(x^2-c)(bx^2-bc-1)$ is separable and a product of two degree-2 irreducible factors over $K$ and $L.$ According to \cite[Proposition 7.1.1]{Sk01}, the Brauer group $\Br(V_2)/\Br(K)\cong\Br(V_{2L})/\Br(L)\cong \ZZ/2\ZZ.$ Furthermore, by Proposition 7.1.2 in loc. cit, we take the quaternion algebra $A=(a,x^2-c)\in \Br(V_2)$ as a generator element of this group. Then we have the equality $A=(a,x^2-c)=(a,bx^2-bc-1)$ in $\Br(V_2).$ 
		
		Thirdly, We need to compute the evaluation of $A$ on $V_2(K_v)$ for all $v\in \Omega_K.$ 
		
		By Remark \ref{remark the implicit function thm and local constant Brauer group}, it suffices to compute the local invariant $\invap$ for all $P_v\in V_2^0(K_v)$ and all $v\in \Omega_K.$
		
		\begin{enumerate}
		\item Suppose that $v\in (\infty_K\backslash S')\cup 2_K.$  Then $a\in K_v^{\times 2},$ so $\invap=0$ for all $P_v\in V_2(K_v).$
		\item Suppose that $v\in (S'\backslash S)\cap \infty_K.$  By the choice of $b,c,$ we have $\tau_v(b)>0$ and $\tau_v(bc+1)<0.$ So, for any $x\in K,$ we have $(a,bx^2-bc-1)_v=1.$ Hence $\invap=0$ for all $P_v\in V_2^0(K_v).$
		\item Suppose that $v\in S'\backslash (S\cup \infty_K).$ By the choice of $b,$ we have $(a,b)_v=1.$ Take an arbitrary $P_v\in V_2^0(K_v).$
		%	If $v(a)$ is even and $\invap=1/2,$ then $(a,bx^2-bc-1))_v=(a,x^2-c)_v=-1$ at $P_v.$ By Lemma \ref{lemma hilbert symbal lifting for odd prime}, the last equality implies that $v(x^2-c)$ is odd, so it is positive. But $(a,bx_0^2-bc-1))_v=(a,-1)_v=1,$ which is a contradiction. Consider the case $v(a)$ is odd.
		If $v(x)<0$ at $P_v,$ by Lemma \ref{lemma Hensel lemma for Hilbert symbal}, we have  $(a,x^2-c)_v=(a,x^2)_v=1.$ 
		If $v(x)> 0$ at $P_v,$ since $b,c\in \Ocal_{K_v}^\times$  and $v(bc+1)=v(a)+2\geq 3,$ by Lemma \ref{lemma Hensel lemma for Hilbert symbal}, we have $(a,x^2-c)_v=(a,-c)_v.$ By Hensel's lemma, we have $-bc=1-(bc+1)\in K_v^{\times 2}.$
		So $(a,x^2-c)_v=(a,-c)_v=(a,-bc)_v=1.$
		If $v(x)=0$ at $P_v,$ since $b\in \Ocal_{K_v}^\times$ and $v(bc+1)=v(a)+2\geq 3,$ by Lemma \ref{lemma Hensel lemma for Hilbert symbal}, we have $(a,bx^2-bc-1)_v=(a,bx^2)_v=1.$ So $\invap=0.$
		\item Suppose that $v\in \Omega_K^f\backslash ( S'\cup 2_K ).$ Take an arbitrary $P_v\in V_2^0(K_v).$ If $\invap=1/2,$ then $(a,bx^2-bc-1)_v=(a,x^2-c)_v=-1$ at $P_v.$ Since $v(a)$ is even, by Lemma \ref{lemma hilbert symbal lifting for odd prime}, the last equality implies that $v(x^2-c)$ is odd, so it is positive. So $v(bx^2-bc-1)=0.$ By Lemma \ref{lemma hilbert symbal lifting for odd prime}, we have $(a,bx^2-bc-1)_v=1,$ which is a contradiction. So $\invap=0.$		
		\item Suppose that $v\in S\cap\infty_K.$  Take an arbitrary $P_v\in V_2^0(K_v).$ If $A(P_v)=0,$ then $(a,bx^2-bc-1)_v=(a,x^2-c)_v=1$ at $P_v.$ The last equality implies that $\tau_v(x^2-c)>0.$  By the choice of $b,$ we have $\tau_v(b)<0,$ so $\tau_v(bx^2-bc-1)<0,$ which contradicts $(a,bx^2-bc-1)_v=1.$ So $\invap=\half.$
		\item Suppose that $v\in S\backslash\infty_K.$ By the choice of $b,$ we have $(a,b)_v=-1.$ Take an arbitrary $P_v\in V_2^0(K_v).$ If $v(x)\leq 0$ at $P_v,$ for $b\in \Ocal_{K_v}^\times$ and $v(bc+1)=v(a)+2\geq 3,$ by Lemma \ref{lemma Hensel lemma for Hilbert symbal}, we have $(a,bx^2-bc-1)_v=(a,bx^2)_v=-1.$ 
		If $v(x)> 0$ at $P_v,$ since $b,c\in \Ocal_{K_v}^\times$ and $v(bc+1)=v(a)+2\geq 3,$ by Lemma \ref{lemma Hensel lemma for Hilbert symbal}, we have $(a,x^2-c)_v=(a,-c)_v.$
		By Hensel's lemma, we have $-bc=1-(bc+1)\in K_v^{\times 2}.$
		% $1-(bc+1)b^{-1}c^{-1}\in K_v^{\times 2}.$ For $-bc(1-(bc+1)b^{-1}c^{-1})=1,$ we have $-bc\in K_v^{\times 2}.$
		So	$(a,x^2-c)_v=(a,-c)_v=-(a,-bc)_v=-1.$ So $\invap=\half.$									
	\end{enumerate} 

		Finally, we need to compute the evaluation of $A$ on $V_2(L_{v'})$ for all $v'\in \Omega_L.$
		\begin{enumerate}
					\item Suppose that $v' \in S_L.$ Let $v\in \Omega_K$ be the restriction of $v'$ on $K.$ By the assumption that $v$ splits completely in $L,$ we have $K_v=L_{v'}.$ So $V_2(K_v)=V_2(L_{v'}).$ Then for any $P_{v'}\in V_2(L_{v'}),$ denote $P_{v'}$ in $V_2(K_v)$ by $P_v.$ Then by the argument already shown,  the local invariant $\inuap=\invap=\half.$ 
			\item Suppose that $v'\in \Omega_L\backslash S_L.$ This local computation is the same as the case $v\in \Omega_K\backslash S.$
		\end{enumerate}
	\end{proof}

	\begin{remark}\label{remark the valuation of Brauer group on local points are fixed and on S nontrivial}
		If  the surface $V_2$ has a $K$-rational point $Q,$ then by the global reciprocity law, the sum $\sum_{v\in \Omega_K}\inv_v(A(Q))=0$ in $\QQ/\ZZ.$	
		If the number $\sharp S$ is odd, then from Proposition \ref{proposition the valuation of Brauer group on local points are fixed and on S nontrivial}, this sum is ${\sharp S}/2,$ which is nonzero in $\QQ/\ZZ.$ So, in this case, the surface $V_2$ has no $K$-rational point, which implies that the surface $V_2$ is a counterexample to the Hasse principle.	
		If the number $\sharp S$ is even, then for any $(P_v)_{v\in\Omega_K}\in V_2(\AA_K),$ by Proposition \ref{proposition the valuation of Brauer group on local points are fixed and on S nontrivial}, the sum $\sum_{v\in \Omega_K}\inv_v(A(P_v))={\sharp S}/2$ is $0$ in $\QQ/\ZZ.$ For $\Br(V_2)/\Br(K)$ is generated by the element $A,$ we have $V_2(\AA_K)^{\Br}= V_2(\AA_K)\neq \emptyset.$ According to \cite[Theorem B]{CTSSD87a,CTSSD87b}, the Brauer-Manin obstruction to the Hasse principle and weak approximation is the only one for Ch\^atelet surfaces. So, in this case, the set $V_2(K)\neq \emptyset,$ and it is dense in $V_2(\AA_K)^{\Br}= V_2(\AA_K),$ i.e. the surface $V_2$ has a $K$-rational point and satisfies weak approximation.  In particular, if the number $\sharp S=0,$ i.e. $S=\emptyset,$ though the Brauer group $\Br(V_2)/\Br(K)$ is nontrivial, it gives no obstruction to week approximation for $V_2.$
	\end{remark}

	Combining the construction method in Subsection \ref{subsection Choice of parameters for V3} with the global reciprocity law, we can relate the properties in Proposition \ref{proposition the valuation of Brauer group on local points are fixed outside S and take two value on S} to the Hasse principle and weak approximation. 
	
	\begin{thm}\label{thm: interesting result for Chatelet surface2}
	For any extension of number fields $L/K,$ there exists a Ch\^atelet surface $V$ over $K$ with $V(\AA_K)\neq \emptyset,$ such that
	for every intermediate field $K\subset L'\subset L,$ the Brauer group $\Br(V)/\Br(K)\cong\Br(V_{L'})/\Br(L')\cong \ZZ/2\ZZ.$ And the surface $V_{L'}$ has the following properties.
	\begin{itemize}
		\item If the degree $[L':K]$ is odd, then the surface $V_{L'}$ is a counterexample to the Hasse principle, i.e. $V(L')=\emptyset.$ In particular, the surface $V$ is a counterexample to the Hasse principle.
		\item If the degree $[L':K]$ is even, then the surface $V_{L'}$ satisfies weak approximation. In particular, in this case, the set $V(L')\neq \emptyset.$
	\end{itemize}
\end{thm}

	\begin{proof}
		By Theorem \ref{theorem Chebotarev density theorem}, we can take a place $v_0\in \Omega_K\backslash (\infty_K^c\cup 2_K)$ splitting completely in $L.$ Let $S=\{v_0\}.$ Using the construction method in Subsection \ref{subsection Choice of parameters for V3}, there exists a Ch\^atelet surface $V$ defined over $K$ having the properties of Proposition \ref{proposition the valuation of Brauer group on local points are fixed and on S nontrivial}. For any subfield $L'\subset L$ over $K,$ by the same argument as in the proof of Proposition \ref{proposition the valuation of Brauer group on local points are fixed and on S nontrivial}, we have $\Br(V)/\Br(K)\cong\Br(V_{L'})/\Br(L')\cong \ZZ/2\ZZ.$ Since $v_0$ splits completely in $L,$ it also does in $L'.$ Since $\sharp S$ is odd, if $[L':K]$ is odd, then $\sharp S_{L'}$ is odd; if $[L':K]$ is even, then $\sharp S_{L'}$ is even. Applying the same argument about the field $L$ to its subfield $L',$ the properties that we list, are just what we have explained in Remark \ref{remark the valuation of Brauer group on local points are fixed and on S nontrivial}.
	\end{proof}
	
	\begin{remark}
		Though the Brauer group $\Br(V)/\Br(K)\cong\Br(V_{L'})/\Br(L')\cong \ZZ/2\ZZ$ in Theorem \ref{thm: interesting result for Chatelet surface2}, is nontrivial, it gives an obstruction to the Hasse principle for $V,$ also $V_{L'}$ if $[L':K]$ is odd; but no longer gives an obstruction to week approximation for $V_{L'}$ if $[L':K]$ is even.
	\end{remark}

	Using the construction method in Subsection \ref{subsection Choice of parameters for V3}, we have the following example, which is a special case of Proposition \ref{proposition the valuation of Brauer group on local points are fixed and on S nontrivial}. 
	
	\begin{eg}
				Let $K=\QQ,$ and let $\zeta_7$ be a primitive $7$-th root of unity. Let $\alpha=\zeta_7+\zeta_7^{-1}$ with the minimal polynomial $x^3+x^2-2x-1.$ Let  
		$L=\QQ(\alpha).$  Then the degree $[L:K]=3.$ Let $S=\{13\}.$ For $13^2\equiv 1\mod 7,41^2\equiv 1\mod 7,$ and $43\equiv 1\mod 7,$ the places $13,41,43$ split completely in $L.$ Using the construction method in Subsection \ref{subsection Choice of parameters for V3}, we choose data: $S=\{13\},~S'=\{13,29\},~S''=\{5\},~v_1=43,~v_2=41,~a=377,~b=5,~c=878755181$ and  $P(x)=(x^2-878755181)(5x^2-4393775906).$ Then the Ch\^atelet surface given by $y^2-377z^2=P(x),$ has the properties of  Proposition \ref{proposition the valuation of Brauer group on local points are fixed and on S nontrivial}. 
	\end{eg}

	\begin{footnotesize}
		\noindent\textbf{Acknowledgements.} The author would like to thank his thesis advisor Y. Liang for proposing the related problems, papers and many fruitful discussions, and thank  the anonymous
		referees for their careful scrutiny and valuable suggestions.  The author is partially supported by NSFC Grant No. 12071448.
	\end{footnotesize}
	

\end{document}